\documentclass{aims}

\usepackage[latin1]{inputenc}
\usepackage[T1]{fontenc}
\usepackage[american]{babel} 

 \usepackage[colorlinks=true]{hyperref}
\hypersetup{urlcolor=blue, citecolor=red}
\usepackage[capitalise,nameinlink]{cleveref}
\usepackage{enumitem} 

\def\currentvolume{X}
 \def\currentissue{X}
  \def\currentyear{20XX}
   \def\currentmonth{XX}
    \def\ppages{X--XX}
     \def\DOI{10.3934/xx.xx.xx.xx}

\usepackage{amsfonts,amssymb,amsthm,amsopn}
\usepackage[centercolon]{mathtools}
\usepackage{dsfont}

\makeatletter
\def\th@plain{%
	\thm@notefont{}
	\itshape 
}
\def\th@definition{%
	\thm@notefont{}
	\normalfont 
}

\newtheorem{theorem}{Theorem}[section]
\newtheorem{lemma}[theorem]{Lemma}
\newtheorem{proposition}[theorem]{Proposition}
\newtheorem{corollary}[theorem]{Corollary}
\newtheorem{remark}[theorem]{Remark}
\newtheorem{assumption}[theorem]{Assumption}
\newtheorem{definition}[theorem]{Definition}

\DeclareMathOperator{\sgn}{sgn}
\newcommand{\dd}{\mathrm{d}}
\newcommand{\LL}{\mathcal{L}}
\newcommand{\QQ}{\mathcal{Q}}
\newcommand{\RR}{\mathcal{R}}
\newcommand{\Uad}{U_{\textup{ad}}}
\newcommand{\N}{\mathbb{N}}
\newcommand{\R}{\mathbb{R}}
\newcommand{\weakly}{\rightharpoonup}

\title[Optimal Control of Non-Lipschitzian PDEs]{Optimal Control of Semilinear 
Elliptic Partial Differential Equations with Non-Lipschitzian Nonlinearities}

\author[C.\ Christof]{}

\subjclass{49J20, 49K20, 49K40}
\keywords{optimal control, nonsmooth optimization, optimality condition, KKT-system, non-Lipschitzian nonlinearity,
semilinear partial differential equation, porous media flow.}

 \email{constantin.christof@uni-due.de}
 
\thanks{$^*$Corresponding author: Constantin Christof}

\begin{document}
\maketitle

\centerline{\scshape Constantin Christof$^*$}
\medskip
{\footnotesize
\centerline{Faculty of Mathematics, University of Duisburg-Essen,}
\centerline{Thea-Leymann-Str.~9, 45127 Essen, Germany}
}
\medskip

\bigskip

\centerline{(Communicated by the associate editor name)}

\begin{abstract}
We study optimal control problems that are 
governed by semi\-linear elliptic partial differential equations 
that involve non-Lipschitzian non\-linearities.
It is shown that, for a certain class of such PDEs, the solution map is  
Fr\'{e}chet differentiable even though the differential operator contains a nondifferentiable
term.
We exploit this effect to establish first-order necessary optimality conditions 
for minimizers of the considered control problems. 
The resulting KKT-conditions take the form of coupled 
PDE-systems
that are posed in non-Muckenhoupt weighted Sobolev spaces 
and raise interesting questions regarding the 
regularity of optimal controls, the derivation of second-order optimality conditions, 
and the analysis of finite element discretizations.
\end{abstract}

\maketitle


\section{Introduction} 
\label{sec:1}

This paper is concerned with  optimal control problems 
of the following type:
\begin{equation}
\label{eq:P}
\tag{P}
\left.
\begin{aligned}
\text{Minimize } & J(y,u)
\\
\text{w.r.t.\ } &y \in H_0^1(\Omega),\quad u \in L^2(\Omega),
\\
\text{ s.t.\ } & 
-\Delta y 
+ \sgn(y)|y|^{\alpha} = u \text{ in } \Omega,
\quad y = 0 \text{ on } \partial \Omega,
\\
\text{ and } & u \in \Uad.
\end{aligned}
~~\right \}
\end{equation}
Here, $\Omega \subset \R^d$, $d \in \N$, is a nonempty
open bounded set with boundary  
$\partial \Omega$; the spaces $H_0^1(\Omega)$ and
$L^2(\Omega)$ are defined as usual (see \cite{Attouch2006});
\mbox{$J\colon H_0^1(\Omega) \times L^2(\Omega) \to \R$}
is a Fr\'{e}chet differentiable objective function;
$\Delta$ denotes the Laplace operator;
$\sgn\colon \R \to \R$ denotes the signum function;
 $\alpha \in (0,1)$ is a given number;  
 $\Uad$ is a nonempty convex  subset of $L^2(\Omega)$;
 and the governing partial differential equation (PDE) is understood in the weak sense.
 Note that the salient feature of \eqref{eq:P} is 
 that the state equation contains a Nemytskii operator 
 that is neither differentiable nor Lipschitz continuous. 
 The main purpose of this paper is to point out that 
 \eqref{eq:P} nevertheless possesses a Fr\'{e}chet differentiable 
 control-to-state map $S\colon L^2(\Omega) \to H_0^1(\Omega)$, $ u \mapsto y$,
 and, thus, allows for the derivation of
first-order necessary optimality conditions  in qualified form. 
 For the main results of our analysis, we refer the reader to \cref{th:Sgateaux},
\cref{cor:Frechet}, and \cref{th:KKT}. 

Before we begin with our study of the problem \eqref{eq:P},
let us give some background:
Semilinear PDEs involving 
non-Lipschitzian 
nonlinearities of the type $\sgn(y)|y|^\alpha$, $\alpha \in (0,1)$, 
arise, for instance, when modeling chemical reactions
in porous media and processes in
desalination plants; see \cite{Alt1986,Barret1997Part1,Diaz2012,Knabner2017} and the references therein.
In these fields, 
the 
exponent $\alpha$ is also known 
as the order
of the reaction process \cite{Diaz2012},
and
the semilinearity 
$\sgn(y)|y|^\alpha$ 
as the \emph{Freundlich isotherm} \cite{Barret1997Part1}.
A main feature of the nonlinearity $\sgn(y)|y|^\alpha$, $\alpha \in (0,1)$,  is that it promotes 
the formation of so-called \emph{dead zones} within $\Omega$ where the 
PDE-solution $y$ vanishes identically \cite{Barret1991,Diaz2012}. 
The price that one pays for this effect is  that
the term $\sgn(y)|y|^\alpha$ is neither Lipschitz continuous nor  differentiable
and, thus, induces a form of non\-smoothness that is
often hard to handle analytically.
The latter is in particular true 
when 
first-order necessary optimality conditions for 
optimal control problems like \eqref{eq:P} are considered
as these classically 
require at least some form 
of (directional) differentiability 
for the PDE-constraint \cite{Troeltzsch2010}. 
Because of these difficulties, 
optimal control and inverse problems 
governed by partial differential equations 
involving non-Lipschitzian nonlinearities 
are rarely addressed in the literature. 
One of the few contributions on this topic is 
\cite{Diaz2012} which studies the existence of 
optimal controls for minimization problems that are
governed by  PDEs similar to that in \eqref{eq:P}.
Compare also with \cite{Barbu1984} in this context 
where general regularization techniques are considered.

The main goal of the present paper is to point out 
that
problems of the type  \eqref{eq:P} are amenable to 
classical techniques from optimal control theory even in the presence of
the nonsmooth term $\sgn(y)|y|^\alpha$. 
  The key observation is that,
although nonsmooth and non-Lipschitzian,
the state equation of \eqref{eq:P} possesses 
a Fr\'{e}chet differentiable solution map $S\colon L^2(\Omega) \to H_0^1(\Omega)$, $u \mapsto y$;
see \cref{th:Sgateaux,cor:Frechet}.
This causes the reduced objective function
$L^2(\Omega) \ni u \mapsto J(S(u),u) \in \R$ of \eqref{eq:P}
to be Fr\'{e}chet differentiable as well, makes it possible to formulate 
a standard Bouligand stationarity condition for minimizers of 
\eqref{eq:P} (see \cref{prop:Bouligand}),
and allows for the derivation of a
KKT-type optimality system (see \cref{th:KKT}).
Note that these results also imply 
that PDEs of the form $-\Delta y 
+ \sgn(y)|y|^{\alpha} = u$ with $\alpha \in (0,1)$ are, in a certain sense, 
better behaved than semilinear partial differential equations that 
involve a Lipschitz continuous, directionally differentiable 
nonlinearity; cf.\ \cite{Betz2019,Betz2019-2,Christof2021,Christof2018nonsmoothPDE,Dong2024,Meyer2017}.
However, 
the 
non\-smooth\-ness of the PDE governing \eqref{eq:P} 
is not without consequences either.
As the analysis of \cref{sec:5} shows, 
it causes the 
obtained KKT-systems to involve 
Sobolev spaces with singular weights
that depend on the state,
may blow up arbitrarily fast, 
and typically do not satisfy a Muckenhoupt property. 
At least to the best of the author's knowledge, 
this effect and the resulting KKT-conditions have not been documented 
so far in the literature. 
Note that the appearance of the weighted Sobolev spaces 
also causes the derivation of further results, e.g., on finite element error estimates,
quadratic growth conditions, and regularity properties of optimal controls, to be 
a very interesting topic.

We would like to mention at this point that we consider \eqref{eq:P}
as a model problem in this paper. 
It is easy to extend our analysis, for example, to the case where the PDE in
\eqref{eq:P} involves an elliptic  second-order partial differential operator more complicated than $-\Delta$ (asymmetric, nonlinear, etc.); cf.\ the general framework in \cite[Assumption~1.2.1]{ChristofPhd2018}.
It should be noted further that the techniques that we use in 
\cref{sec:4} for analyzing the differentiability properties of 
the control-to-state map $S$ of \eqref{eq:P} are, in fact, part of a larger theoretical 
machinery for the analysis of elliptic variational inequalities (VIs) of the second kind
 that also allows for
generalizations in  various other directions; see \cref{rem:end_sec_4} and \cite{ChristofPhd2018,ChristofWachsmuth2020}. 
This can also be seen in \cref{sec:3}, where
 one of the first steps of our analysis is to reformulate the 
state equation of \eqref{eq:P} as an elliptic VI; see \cref{lem:VIreformulation}. Once this 
reformulation is obtained, our proof of the Fr\'{e}chet differentiability of the 
control-to-state map $S\colon L^2(\Omega) \to H_0^1(\Omega)$, $ u \mapsto y$, of \eqref{eq:P}
relies on the observation that---although non-smooth and non-Lipschitzian---the semilinearity 
in the PDE $-\Delta y + \sgn(y)|y|^{\alpha} = u$ cannot cause the difference quotients of $S$ to become unbounded in $H_0^1(\Omega)$ as it constitutes a monotone part of the 
differential operator and, as such, produces terms with the ``right'' sign in a-priori estimates; cf.\ the Lipschitz continuity result in \cref{lemma:uniquesolvability}. The resulting boundedness of the difference 
quotients of $S$  makes it possible to extract weakly convergent subsequences and, by 
means of a passage to the limit and a careful analysis of integrability properties,
to derive a variational problem that, on the one hand, 
has to be solved by all weak accumulation points of the difference quotients
and, on the other hand, can have at most one solution---thus establishing the weak 
directional differentiability of $S$.
The structure of the limit problem and compactness arguments 
then yield the G\^{a}teaux differentiability
of $S$ and, ultimately, the 
asserted Fr\'{e}chet differentiability property; see \cref{cor:Frechet}. 
We remark that, using the theory of second-order epi-derivatives, 
one can characterize precisely in which situations the above 
line of reasoning allows to prove the G\^{a}teaux differentiability
of the solution operator of an elliptic VI/an elliptic semilinear PDE;
see \cite[Corollary~1.4.4]{ChristofPhd2018} and \cref{rem:end_sec_4}. 
However, the conditions that arise in this context are
very abstract and hard to verify in practice.
The present paper shows that,
for semilinear elliptic PDEs of the form $-\Delta y + \sgn(y)|y|^{\alpha} = u$,
these abstract conditions are satisfied. 
Even more, 
we illustrate that, for this type of problem, 
the Fr\'{e}chet differentiability of the solution map
can be proved by means of rather elementary arguments
that do not require notions of generalized second-order differentiability 
and well visualize the underlying mechanisms and effects.

We conclude this introduction with an overview of the content and the structure 
of the remainder of the paper.

\Cref{sec:2,sec:3} are concerned with preliminaries. 
Here, we clarify the \mbox{notation} and recall basic results on the 
state equation of \eqref{eq:P}. 
In \cref{sec:4}, we establish the Fr\'{e}chet differentiability of 
the control-to-state operator $S\colon L^2(\Omega) \to H_0^1(\Omega)$,
$u \mapsto y$. The proof that we use in this section relies 
on
Lipschitz continuity properties of $S$,
 a reformulation of the PDE $-\Delta y 
+ \sgn(y)|y|^{\alpha} = u$ as an elliptic variational inequality of the second kind,
and 
techniques that have been developed in \cite{ChristofPhd2018,ChristofWachsmuth2020}.
\Cref{sec:5} is concerned with the derivation of 
first-order necessary 
optimality conditions for local minimizers of  
\eqref{eq:P}. Here, we also show that the obtained KKT-systems and 
Stampacchia truncation arguments allow 
to establish improved 
$L^q(\Omega)$-regularity properties for optimal controls 
of \eqref{eq:P}. The paper concludes with some remarks on 
open questions and topics for future research. 

\section{Basic notation}
\label{sec:2}

In what follows, we
use the symbols $\|\cdot\|$ and $(\cdot,\cdot)$
to denote 
norms and inner products on real vector spaces,
respectively, equipped with a subscript that clarifies the space under consideration.
For the Euclidean norm on $\R^d$, $d \in \N$,
we write $|\cdot|$.
The space of linear and continuous functions 
from a normed space $(X, \|\cdot\|_X)$
to a normed space $(Y, \|\cdot \|_Y)$
is denoted by $\LL(X,Y)$. In the special case $Y=\R$,
we write $X^* := \LL(X,\R)$ for the topological dual space of $X$
and $\langle \cdot, \cdot \rangle_X$ for a dual pairing, 
i.e., $\langle x^*, x \rangle_X :=  x^*(x)$ for all $x^* \in X^*$, $x \in X$. 
The modes of weak and strong convergence 
in a normed space $(X, \|\cdot\|_X)$
are denoted by $\weakly$ and $\to$, respectively. 
If $(X,\|\cdot\|_X)$ embeds continuously into $(Y, \|\cdot\|_Y)$,
then we write $X \hookrightarrow Y$. 
Recall that a function $F\colon X \to Y$ between 
normed spaces $(X, \|\cdot\|_X)$
and $(Y, \|\cdot \|_Y)$ is called G\^{a}teaux differentiable 
if the directional derivative 
\[
F'(x;h) :=
\lim_{0 < \tau \to 0} \frac{F(x + \tau h) - F(x)}{\tau}
\]
exists for all $x,h \in X$ and is linear and continuous in $h$. 
In this case, the operator $F'(x;\cdot) \in \LL(X,Y)$ is called the 
G\^{a}teaux derivative of $F$ at $x$ and denoted by $F'(x)$. If,
for every $x \in X$, there exists $F'(x) \in \LL(X,Y)$ such that 
\[
\lim_{0 < \|h\|_X \to 0} \frac{\|F(x+h) - F(x) - F'(x)h\|_Y}{\|h\|_X} = 0
\]
holds, 
then $F$ is called Fr\'{e}chet differentiable with derivative 
$F'\colon X\to \LL(X,Y)$. If the argument of $F$ has several components, 
then a partial Fr\'{e}chet derivative with respect to (w.r.t.)  a component $z$ 
is denoted by $\partial_z$.

Given a Lebesgue measurable  set $\Omega \subset \R^d$, $d \in \N$,
we denote by 
$L^q(\Omega)$, $1 \leq q \leq \infty$, the standard real Lebesgue spaces on $\Omega$,
equipped with their usual norms. 
For level sets and complements of level sets
of  elements $v$ of the Lebesgue 
spaces (defined up to sets of measure zero), we use the 
shorthand notation $\{v = c\}$ and $\{v \neq c\}$, $c \in \R$, respectively. 
The $\{0,1\}$-indicator function of a measurable set $D \subset \Omega$ is
denoted by $\mathds{1}_D \in L^\infty(\Omega)$.
If $\Omega$ is a nonempty, open, and bounded set, then we write 
$W^{k,q}(\Omega)$, 
$H^k(\Omega)$, $k \in \N$, $1\leq q \leq \infty$,
for the Sobolev spaces on $\Omega$, defined as in \cite[Chapter~5]{Attouch2006} and equipped 
with their usual norms.
With $H_0^1(\Omega)$, we denote the closure of the 
set $C_c^\infty(\Omega)$ of smooth functions with compact support on $\Omega$
w.r.t.\ the 
$H^1(\Omega)$-norm. 
Recall that the space $H_0^1(\Omega)$ is Hilbert when endowed with the inner product
\[
\left ( v, w\right )_{H_0^1(\Omega)}
:=
\left \langle -\Delta v, w\right \rangle_{H_0^1(\Omega)}
=
\int_\Omega \nabla v \cdot \nabla w\,\dd x \qquad \forall v,w \in H_0^1(\Omega).
\]
Here, $\Delta$ and $\nabla$ are the distributional Laplacian and the weak gradient, respectively. 
As usual,  we denote by $H^{-1}(\Omega)$ the dual space of 
$H_0^1(\Omega)$ with pivot space $L^2(\Omega)$, i.e., 
we use the identifications 
$H_0^1(\Omega) \hookrightarrow L^2(\Omega) \cong L^2(\Omega)^* \hookrightarrow H^{-1}(\Omega)$.
Recall that the Sobolev embeddings imply that we have
\begin{equation}
\label{eq:sob_emb_primal}
H_0^1(\Omega)
\hookrightarrow 
L^q(\Omega)
\qquad 
\forall 
q\in 
\QQ_d
:=
\begin{cases}
[1,\infty] & \text{ if } d=1,
\\
[1,\infty) & \text{ if } d=2,
\\
\displaystyle
\left [
1, \frac{2d}{d-2}  \right ]
& \text{ if } d > 2,
\end{cases}
\end{equation}
and, by duality, 
\begin{equation}
\label{eq:sob_emb_dual}
L^q(\Omega)
\hookrightarrow 
H^{-1}(\Omega)
\qquad 
\forall 
q\in 
\QQ_d^*
:=
\begin{cases}
[1,\infty] & \text{ if } d=1,
\\
(1,\infty] & \text{ if } d=2,
\\
\displaystyle
\left [\frac{2d}{d+2}, \infty  \right ]
& \text{ if } d > 2;
\end{cases}
\end{equation}
see \cite[Theorem 5.7.2]{Attouch2006}.
Recall further that the embeddings in \eqref{eq:sob_emb_primal}
are compact except for the limit case  $q= 2d/(d-2)$, $d>2$;
see \cite[Theorem 5.7.7]{Attouch2006}. 
(Note that no regularity of the boundary $\partial \Omega$ is needed here 
since the space $H_0^1(\Omega)$ is considered; cf.\ \cite[Theorem 7.22]{Gilbarg2001elliptic}
and the comments thereafter.)
By Schauder's theorem, this also implies that the 
embeddings in \eqref{eq:sob_emb_dual} are compact, 
except for the case $q=2d/(d+2)$, $d>2$.
We remark that additional symbols etc.\
are introduced in the remainder of this paper wherever necessary. 
This notation is clarified on its first appearance.

\section{Well-posedness and Lipschitz stability of the governing PDE}
\label{sec:3}
We begin our study of the optimal control problem \eqref{eq:P} by 
collecting preliminary results on the well-posedness and Lipschitz stability 
of the state equation
\begin{equation}
\tag{D}
\label{eq:govPDE}
-\Delta y  
+  \sgn(y)|y|^{\alpha} = u \text{ in } \Omega,
\quad y = 0 \text{ on } \partial \Omega.
\end{equation}
For the sake of clarity, we restate 
our assumptions on the quantities in \eqref{eq:govPDE} in:

\begin{assumption}[standing assumptions on the governing PDE]~
\begin{enumerate}[label=\roman*)]
\item $\Omega \subset \R^d$, $d \in \N$, is a nonempty
open bounded set;
\item $\alpha \in (0,1)$ is a given exponent;
\item $u \in H^{-1}(\Omega)$ is a given right-hand side.
\end{enumerate}
\end{assumption}

As usual, we understand 
solutions of \eqref{eq:govPDE} in the weak sense. 

\begin{definition}[weak solution]\label{def:weak_solution}%
Given $u \in H^{-1}(\Omega)$,
we call $y$ a weak solution of \eqref{eq:govPDE} 
with right-hand side $u$ if 
$y \in H_0^1(\Omega)$ holds and 
\begin{equation}
\label{eq:weak_form_D}
\int_\Omega \nabla y \cdot \nabla v 
+
\sgn(y)|y|^{\alpha}
v\,
\dd x
=
\langle u,v \rangle_{H_0^1(\Omega)}
\quad 
\forall v \in H_0^1(\Omega).
\end{equation}
\end{definition}

The existence of a unique weak solution of \eqref{eq:govPDE}
for all $u \in H^{-1}(\Omega)$
is a straightforward consequence of 
standard results from monotone operator theory. 

\begin{lemma}[unique solvability of (D)]
\label{lemma:uniquesolvability}
The PDE \eqref{eq:govPDE} possesses a unique weak solution $y \in H_0^1(\Omega)$
for all $u \in H^{-1}(\Omega)$. If, further, $y_1, y_2 \in H_0^1(\Omega)$
solve \eqref{eq:govPDE} with right-hand sides 
$u_1, u_2 \in H^{-1}(\Omega)$, respectively, then the following 
stability estimate holds true:
\begin{equation}
\label{eq:SLipschitz}
\|y_1 - y_2\|_{H_0^1(\Omega)}\leq \|u_1 - u_2\|_{H^{-1}(\Omega)}.
\end{equation}
\end{lemma}

\begin{proof}
This follows from 
\cite[Theorem 4.1]{Troeltzsch2010} and trivial calculations. 
\end{proof}

Due to 
\cref{lemma:uniquesolvability}, it makes sense to introduce the following definition.

\begin{definition}[solution operator]
We denote by $S\colon H^{-1}(\Omega) \to H_0^1(\Omega)$, $u \mapsto y$,
the solution operator of the partial differential equation \eqref{eq:govPDE}.
\end{definition}

Note that, due to \eqref{eq:SLipschitz}, $S$ is globally one-Lipschitz, i.e., 
\begin{equation}
\label{eq:SLipschitz-SS}
\|S(u_1) - S(u_2)\|_{H_0^1(\Omega)}\leq \|u_1 - u_2\|_{H^{-1}(\Omega)}
\qquad
\forall u_1,u_2 \in H^{-1}(\Omega).
\end{equation}
We remark that, in the remainder of this paper,
the operator $S$ is sometimes also considered on domains of definition that are smaller than 
$H^{-1}(\Omega)$ (e.g., $L^2(\Omega)$ in the context of the problem \eqref{eq:P}). 
We denote these restrictions of $S$ with the same symbol 
for the sake of simplicity. The next lemma establishes 
a connection between the PDE 
\eqref{eq:govPDE} and a certain elliptic variational 
inequality of the second kind.

\begin{lemma}[equivalence to an elliptic VI]
\label{lem:VIreformulation}
Let $u \in H^{-1}(\Omega)$ be given. 
A function $y \in H_0^1(\Omega)$ is the weak solution of \eqref{eq:govPDE}
with right-hand side $u$ 
if and only if it solves the elliptic variational inequality of the second kind
\begin{equation}
\label{eq:VI}
\tag{V}
\left.
\begin{aligned}
y \in H_0^1(\Omega),\hspace{0.5cm}&
\\
( y , v - y)_{H_0^1(\Omega)}
&+
\int_\Omega
\frac{|v|^{\alpha+1}}{\alpha+1}
\dd x
- 
\int_\Omega
\frac{|y|^{\alpha+1}}{\alpha+1}
\dd x
\geq
\langle u,v -y\rangle_{H_0^1(\Omega)}
\\
&\hspace{5.4cm}
\forall v \in H_0^1(\Omega).
\end{aligned}
~~
\right\}
\end{equation}
\end{lemma}

\begin{proof}
If $y \in H_0^1(\Omega)$ is the weak solution of \eqref{eq:govPDE},
then our definition of the $H_0^1(\Omega)$-inner product, 
\eqref{eq:weak_form_D}, and
 the properties of convex subgradients imply 
\begin{equation*}
\begin{aligned}
- ( y , v - y)_{H_0^1(\Omega)} + \langle u,v -y\rangle_{H_0^1(\Omega)}
& =
\int_\Omega
\sgn(y)|y|^{\alpha}
(v -y)
\dd x
\\
&\leq
\int_\Omega
\frac{|v|^{\alpha+1}}{\alpha+1}
\dd x
-
\int_\Omega
\frac{|y|^{\alpha+1}}{\alpha+1}
\dd x
\quad \forall v \in H_0^1(\Omega).
\end{aligned}
\end{equation*}
Thus, $y$ is a solution of \eqref{eq:VI} as claimed. 
Note that this also shows that  \eqref{eq:VI} possesses at least one 
solution for all $u \in H^{-1}(\Omega)$;
see \cref{lemma:uniquesolvability}. 
As \eqref{eq:VI}  can have at most one solution
(as one may easily check by means of a trivial contradiction argument;
see \cite[Theorem 4.1]{Glowinski2013}),
it follows that every solution of \eqref{eq:VI} 
coincides with the unique weak solution of \eqref{eq:govPDE}.
This completes the proof. 
\end{proof} 

The main advantage of  the variational inequality
\eqref{eq:VI}  is that it does not involve 
the signum function or potentially singular fractions 
(as appearing in \eqref{eq:govPDE} when using the 
reformulation $\sgn(y)|y|^\alpha = y / |y|^{1-\alpha}$). 
We remark that, using the binomial identities, one can 
easily check that \eqref{eq:VI} can also be recast as an optimization problem 
with a strongly convex objective function. We work with the 
formulation \eqref{eq:VI} in our analysis
because this approach is also feasible when starting with a 
PDE that involves an asymmetric second-order differential operator;
cf.\ the comments on possible extensions of our analysis 
in \cref{sec:1} and  \cref{rem:end_sec_4}.

\section{Differentiability of the PDE solution map}
\label{sec:4}
With the preliminaries in place, 
we can turn our attention to the analysis of the 
differentiability properties of the
solution map $S\colon u \mapsto y$ of
\eqref{eq:govPDE}. To prove that this map 
possesses classical derivatives 
even though \eqref{eq:govPDE} contains terms  that are merely 
H\"{o}lder continuous, 
we will exploit the characterization of $S$ by means 
of  \eqref{eq:VI}
and techniques that have been developed for the 
analysis of elliptic VIs
of the first and the second kind in \cite{ChristofPhd2018,ChristofWachsmuth2020}. 
We begin with a technical lemma that is a 
special version of \cite[Lemma 4.2.2]{ChristofPhd2018}.

\begin{lemma}[Taylor-like expansion for powers of absolute value functions]
\label{lemma:technical_1}
Let $\beta \in (0,1)$. Then, for all $x \in \R \setminus \{0\}$,
 $t> 0$, and  $z \in \R$, we have
\begin{equation}
\label{eq:tech1}
\begin{aligned}
&\frac{1}{t} \left ( \frac{|x + t z |^{\beta+1}  - |x|^{\beta+1} }{t} - (\beta +1) |x|^{\beta - 1} 
x z  \right ) \\
&\quad =  (\beta+1)| x |^{\beta} \frac{1}{t}\left ( \frac{   | x + t z |  - | x |   }{t} - |x| ^{-1} x z  \right ) \\
&\quad\qquad  +\frac{1}{t^2}  ( | x + t z | - |  x |   )^2 
(\beta^2 + \beta)\int_{0}^{1} ( 1 -s)  \Big ( (1-s) | x |  + s | x + t z |  \Big )^{\beta - 1} \dd s.
\end{aligned}
\end{equation}
Further, it holds
\begin{equation}
\label{eq:tech2}
 \frac{1}{t} \left ( \frac{|x + t z| - |x|}{t} -   |x|^{- 1} x z \right ) 
 = \frac{ |x|(|x + t z|   - |x|)  - t   x z  }{|x| (|x + t z| + |x|)^2 } z^2
\end{equation}
and
\begin{equation}
\label{eq:tech3}
\begin{aligned}
& \frac{1}{t^2} (|x + t z|- |x| )^2 
=  \frac{ 4 x^2   }{(|x| + |x + t z|)^2} z^2 +  \frac{ 4t x z    +   t^2 z^2 }{(|x| + |x + t z|)^2}z^2.
\end{aligned}
\end{equation}
\end{lemma}

\begin{proof}
As
the function $(0,c) \ni s \mapsto s^{\beta +1} \in \R$ 
is in $W^{2,q}(0,c)$ for all $1 \leq q < (1-\beta)^{-1}$ and all $c>0$,
Taylor's theorem implies 
\begin{align*}
b^{\beta+1}  - a^{\beta+1}
=
(\beta+1) a^{\beta}(b - a)
+
(b-a)^2 (\beta^2 + \beta)
\int_0^1 (1-s )( (1-s)a + sb )^{\beta - 1} \dd s 
\end{align*}
for all $a >0$ and $b \geq 0$.
If we choose $a := |x| > 0$ and $b := |x + t z | \geq 0$ in the above and plug into 
the left-hand side of \eqref{eq:tech1}, then \eqref{eq:tech1}
follows immediately. The equalities in \eqref{eq:tech2} and \eqref{eq:tech3}
are obtained by straightforward calculation. Indeed, 
by distinguishing between the cases 
$\sgn(x + t z) = \sgn(x)$ and $\sgn(x + t z) \neq \sgn(x)$,
one easily checks that 
\[
 \frac{1}{t} \left ( \frac{|x + t z| - |x|}{t} -   |x|^{- 1} x z \right ) 
 = 
\left ( |x + t z|   - |x|  - t  |x|^{-1}  x z \right)
 \frac{z^2}{(|x + t z| + |x|)^2}
\]
holds for all $x \in \R \setminus \{0\}$,
 $t> 0$, and  $z \in \R$. This proves \eqref{eq:tech2}.
Using the third binomial identity, we may further compute that 
\begin{align*}
 \frac{1}{t^2} (|x + t z|- |x| )^2 
&=   \frac{(|x + t z|^2- |x|^2 )^2}{t^2}  \frac{1}{(|x + t z|+ |x| )^2}
\\
& =   \frac{(2 t x z + t^2 z^2 )^2}{t^2}  \frac{1}{(|x + t z|+ |x| )^2}
\\
& =   \frac{4 x^2 z^2 + 4 t x z^3 + t^2  z^4}{(|x + t z|+ |x| )^2}
\\
& =
\frac{ 4 x^2   }{(|x| + |x + t z|)^2} z^2 +  \frac{ 4t x z    +   t^2 z^2 }{(|x| + |x + t z|)^2}z^2.
\end{align*}
This establishes \eqref{eq:tech3}
and completes the proof of the lemma.
\end{proof}

To prepare our sensitivity analysis, we next 
introduce some notation.

\begin{definition}[difference quotient]
Given $u, h \in H^{-1}(\Omega)$ and $\tau > 0$,
we denote by $\delta_\tau \in H_0^1(\Omega)$ the difference quotient 
\begin{equation}
\label{eq:diffquot_def}
\delta_\tau := \frac{S(u + \tau h) - S(u)}{\tau}. 
\end{equation}
\end{definition}

By exploiting \cref{lem:VIreformulation},
it is easy to establish that the difference quotients 
of $S$ are also characterized by an elliptic variational inequality of the 
second kind.
 
\begin{lemma}[VI for the difference quotients]
For every $u\in H^{-1}(\Omega)$ with state $y := S(u)$, every 
$h \in H^{-1}(\Omega)$,  and every $\tau > 0$, the difference quotient
$\delta_\tau$ is the (necessarily unique) solution of the 
variational inequality 
\begin{equation}
\label{eq:diffquotVI}
\begin{aligned}
&\delta_\tau \in H_0^1(\Omega), 
\\
&( \delta_\tau , z - \delta_\tau)_{H_0^1(\Omega)}
\\
&\quad+
\frac{1}{(\alpha+1)}
\int_\Omega
\frac{1}{\tau}
\left (
\frac{| y + \tau z|^{\alpha +1} - | y|^{\alpha+1}}{\tau}
-
(\alpha+1)
\sgn(y)|y|^{\alpha}  z
\right )
\dd x
\\
&\quad-
\frac{1}{(\alpha+1)}
\int_\Omega
\frac{1}{\tau}
\left (
\frac{| y + \tau \delta_\tau|^{\alpha +1} - | y|^{\alpha+1}}{\tau}
-
(\alpha+1)
\sgn(y)|y|^{\alpha}  \delta_\tau
\right )
\dd x
\\
&\hspace{5.5cm}
\geq 
\langle h,z -\delta_\tau\rangle_{H_0^1(\Omega)}
\quad 
\forall z \in H_0^1(\Omega).
\end{aligned}
\end{equation}
\end{lemma}

\begin{proof}
To show that $\delta_\tau$ solves \eqref{eq:diffquotVI},
one chooses test functions 
of the form $v = y + \tau z$ in the VI \eqref{eq:VI}
satisfied by $S(u + \tau h)$,
plugs in the identity 
$S(u + \tau h) = y + \tau \delta_\tau$ obtained from 
\eqref{eq:diffquot_def},
exploits the variational identity \eqref{eq:weak_form_D} satisfied by $y$,
divides the resulting inequality by $\tau^2$,
and artificially introduces the 
$|y|^{\alpha+1}$-terms 
(which add up to zero in \eqref{eq:diffquotVI}). 
That \eqref{eq:diffquotVI} can have at most one solution 
again follows from a trivial contradiction argument;
see \cite[Theorem 4.1]{Glowinski2013}.
\end{proof}

Note that the integrands in \eqref{eq:diffquotVI}
have the same form as the left-hand side 
of \eqref{eq:tech1}. From \eqref{eq:SLipschitz-SS},
we further obtain that the
family of difference quotients $\{\delta_\tau\}$
associated with a right-hand side $u \in H^{-1}(\Omega)$ and 
a perturbation $h \in H^{-1}(\Omega)$ is bounded in 
$H_0^1(\Omega)$ and, thus, possesses 
weakly convergent subsequences for $\tau \to 0$.
In what follows, the main idea is to pass to the limit $\tau \to 0$ 
in \eqref{eq:diffquotVI} along 
weakly convergent subsequences of difference quotients
to arrive at a limit variational inequality and to exploit 
this limit VI to show that the 
difference quotients $\delta_\tau$ can have only one 
weak accumulation point in $H_0^1(\Omega)$ for $\tau \to 0$. 
By means of a trivial contradiction argument, it then follows that 
the whole family of difference quotients $\delta_\tau$ 
converges weakly in $H_0^1(\Omega)$ to 
a unique $\delta \in H_0^1(\Omega)$ for $\tau \to 0$,
and by means of bootstrapping and compactness arguments,
that the solution operator $S$ of \eqref{eq:govPDE} is Fr\'{e}chet differentiable 
as a function $S\colon L^2(\Omega) \to H_0^1(\Omega)$.
The most delicate point when arguing along these lines is, 
of course, the limit transition with the 
nonsmooth terms in \eqref{eq:diffquotVI}. 
We proceed in several steps, starting with the following lemma.
 
\begin{lemma}[behavior of the $\boldsymbol{\alpha}$-$\boldsymbol{\delta_\tau}$-terms]
\label{lem:lim_transition_1}
Let $u, h \in H^{-1}(\Omega)$ be given. Define $y := S(u)$. 
Suppose that $\delta_k := \delta_{\tau_k}$
are  difference quotients as in \eqref{eq:diffquot_def} 
associated with a sequence  $\{\tau_k\} \subset (0,\infty)$.
Assume that 
$\tau_k \to 0$ holds and that $\delta_k \weakly \delta$ in $H_0^1(\Omega)$ for some 
$\delta \in H_0^1(\Omega)$. Then it holds 
\begin{equation}
\label{eq:delta_vanish}
 \delta = 0 \text{ a.e.\ in } \{y = 0\}
\end{equation}
and
\begin{equation}
\label{eq:delta_liminf}
\begin{aligned}
\infty &> \liminf_{k \to \infty}
\left [
\frac{1}{(\alpha+1)}
\int_\Omega
\frac{1}{\tau_k}
\left (
\frac{| y + \tau_k \delta_k|^{\alpha +1} - | y|^{\alpha+1}}{\tau_k}
-
(\alpha+1)
\sgn(y)|y|^{\alpha}  \delta_k
\right )
\dd x
\right ]
\\
&\geq 
\frac{\alpha}{2}
\int_{\{y \neq 0\}}
\frac{\delta^2}{|y|^{1 - \alpha}} \dd x.
\end{aligned}
\end{equation}
\end{lemma}

\begin{proof}
By choosing the test function $z=0$ in the VI \eqref{eq:diffquotVI} for $\delta_k$, 
we obtain that 
\begin{equation}
\label{eq:longestimate}
\begin{aligned}
&\langle h, \delta_k\rangle_{H_0^1(\Omega)}
\\
&\quad\geq 
\langle h, \delta_k\rangle_{H_0^1(\Omega)}
-
\| \delta_k \| _{H_0^1(\Omega)}^2
\\
&\quad\geq 
\frac{1}{(\alpha+1)}
\int_\Omega
\frac{1}{\tau_k}
\left (
\frac{| y + \tau_k \delta_k|^{\alpha +1} - | y|^{\alpha+1}}{\tau_k}
-
(\alpha+1)
\sgn(y)|y|^{\alpha}  \delta_k
\right )
\dd x
\\
&\quad
=
\frac{1}{(\alpha+1)}
\int_{\{y \neq 0\}}
\frac{1}{\tau_k}
\left (
\frac{| y + \tau_k \delta_k|^{\alpha +1} - | y|^{\alpha+1}}{\tau_k}
-
(\alpha+1)
\sgn(y)|y|^{\alpha}  \delta_k
\right )
\dd x
\\
&\qquad\quad + 
\frac{\tau_k^{\alpha-1}}{(\alpha+1)}
\int_{\{y = 0\}}
 | \delta_k|^{\alpha +1} 
\dd x.
\end{aligned}
\end{equation}
Note that the left-hand side of \eqref{eq:longestimate} is bounded. 
This establishes the first inequality in 
\eqref{eq:delta_liminf}. As difference quotients of 
convex functions dominate the derivatives that 
they approximate, we further have that the integrands 
in the integrals in \eqref{eq:longestimate} are all nonnegative. 
In combination with the convergence $\delta_k \to \delta$ in $L^2(\Omega)$
obtained from the compactness of the embedding $H_0^1(\Omega) \hookrightarrow L^2(\Omega)$,
this yields
\[
0
= 
\limsup_{k \to \infty}\,
(\alpha+1)\tau_k^{1-\alpha}
\langle h, \delta_k\rangle_{H_0^1(\Omega)}
\geq
\limsup_{k \to \infty}
\int_{\{y = 0\}}
 | \delta_k|^{\alpha +1} 
\dd x
=
\int_{\{y = 0\}}
 | \delta |^{\alpha +1} 
\dd x.
\]
Thus, $\delta$ satisfies \eqref{eq:delta_vanish} as claimed.
By revisiting  \eqref{eq:longestimate}, 
by noting that the first term on the 
right-hand side of \eqref{eq:tech1} 
is nonnegative (again due to the properties of difference 
quotients of convex functions),
and by invoking Fatou's lemma
\cite[Corollary~2.8.4]{Bocharev2007}
(keeping in mind that all of the 
involved integrands are nonnegative and 
that  $\delta_k \to \delta$ holds in $L^2(\Omega)$),
we further obtain that
\begin{equation*}
\begin{aligned}
&\liminf_{k \to \infty}
\left [
\int_\Omega
\frac{1}{\tau_k}
\left (
\frac{| y + \tau_k \delta_k|^{\alpha +1} - | y|^{\alpha+1}}{\tau_k}
-
(\alpha+1)
\sgn(y)|y|^{\alpha}  \delta_k
\right )
\dd x
\right ]
\\
&\geq 
\liminf_{k \to \infty}
\left [
\int_{\{y \neq 0\}}
\frac{1}{\tau_k}
\left (
\frac{| y + \tau_k \delta_k|^{\alpha +1} - | y|^{\alpha+1}}{\tau_k}
-
(\alpha+1)
\sgn(y)|y|^{\alpha}  \delta_k
\right )
\dd x
\right ]
\\
&\geq 
\liminf_{k \to \infty}
\Bigg [
\int_{\{y \neq 0\}}
\left ( \frac{| y + \tau_k \delta_k | - |  y | }{\tau_k}  \right )^2
\\
&\hspace{3cm}
\cdot (\alpha^2 + \alpha)
\int_{0}^{1} ( 1 -s)  \Big ( (1-s) | y |  + s | y + \tau_k \delta_k|  \Big )^{\alpha - 1} \dd s
\dd x
\Bigg]
\\
&\geq
\frac{\alpha^2 + \alpha}{2}
\int_{\{y \neq 0\}}
\frac{\delta^2}{| y |^{1-\alpha}}
\dd x.
\end{aligned}
\end{equation*}
This completes the proof. 
\end{proof}

\Cref{lem:lim_transition_1} motivates the following definition.

\begin{definition}[weighted Sobolev space]
Given a right-hand side 
$u \in H^{-1}(\Omega)$ with associated state
$y := S(u)$, we denote by $V_y$ the subspace of $H_0^1(\Omega)$ given by 
\begin{equation*}
\begin{aligned}
V_y
:=
\Bigg \{
v \in H_0^1(\Omega)
\colon
\int_{\{y  \neq 0\}} \frac{v^2}{|y|^{1 - \alpha}} \dd x
< \infty
~
\text{and }
v=0 \text{ a.e.\ in } \{y  = 0\}
\Bigg \}.
\end{aligned}
\end{equation*}
We further define 
$(\cdot, \cdot)_{V_y}$ to be the bilinear form 
\begin{equation*}
(\cdot, \cdot)_{V_y}\colon V_y \times V_y \to \R,
\qquad 
(v, w)_{V_y}
:=
(v,w)_{H_0^1(\Omega)}
+
\alpha
\int_{\{y  \neq 0\}} \frac{v w}{|y|^{1 - \alpha}} \dd x.
\end{equation*}
\end{definition}

Note that $V_y$ is Hilbert when endowed with 
$(\cdot, \cdot)_{V_y}$ as the inner product and that the limit $\delta$ 
in \cref{lem:lim_transition_1} satisfies 
$\delta \in V_y$ by \eqref{eq:delta_vanish} and \eqref{eq:delta_liminf}.
The next lemma shows how to handle 
the limit transition with the 
nonsmooth terms involving the test function $z$ in 
the variational inequality \eqref{eq:diffquotVI}.

\begin{lemma}[behavior of the $\boldsymbol{\alpha}$-$\boldsymbol{z}$-terms]
\label{lem:lim_transition_2}
Let $u \in H^{-1}(\Omega)$ be given. Define $y := S(u)$. 
Suppose that $\{\tau_k\} \subset (0,\infty)$ is a sequence 
satisfying 
$\tau_k \to 0$. Then, for all $z \in V_y$, we have
\begin{equation*}
\begin{aligned}
&\lim_{k \to \infty}
\frac{1}{(\alpha+1)}
\int_\Omega
\frac{1}{\tau_k}
\left (
\frac{| y + \tau_k z|^{\alpha +1} - | y|^{\alpha+1}}{\tau_k}
-
(\alpha+1)
\sgn(y)|y|^{\alpha}  z
\right )
\dd x
\\
&
=
\frac{\alpha}{2}
\int_{\{y \neq 0\}}
\frac{z^2}{|y|^{1-\alpha}} 
\dd x.
\end{aligned}
\end{equation*}
\end{lemma}

\begin{proof}
Due to the identities in \cref{lemma:technical_1}, it holds
\begin{equation*}
\begin{aligned}
&\frac{1}{\tau_k}
\left (
\frac{| y + \tau_k z|^{\alpha +1} - | y|^{\alpha+1}}{\tau_k}
-
(\alpha+1)
\sgn(y)|y|^{\alpha}  z
\right )
\\
&= 
 (\alpha+1)
 \frac{ |y|(|y + \tau_k z|   - |y|)  - \tau_k   y z  }{(|y + \tau_k z| + |y|)^2 } \frac{z^2}{| y |^{1-\alpha}}
 \\
 &
 +
 (\alpha^2 + \alpha)
 \left (
 \frac{ 4 y^2  }{(|y| + |y + \tau_k z|)^2}
 \right )  
\int_{0}^{1} ( 1 -s)  \left( \frac{| y |}{(1-s) | y |  + s | y + \tau_k z |}   \right)^{1-\alpha} \dd s \, \frac{z^2}{| y |^{1-\alpha}}
 \\
 &
 +
 (\alpha^2 + \alpha)
 \left (
  \frac{ 4\tau_k y z    +   \tau_k^2 z^2 }{(|y| + |y + \tau_k z|)^2}
 \right ) \int_{0}^{1} ( 1 -s)  \left( \frac{| y |}{(1-s) | y |  + s | y + \tau_k z |}   \right)^{1-\alpha} \dd s \, \frac{z^2}{| y |^{1-\alpha}}
\end{aligned}
\end{equation*}
a.e.\ in $\{y \neq 0\}$. 
From the triangle inequality and elementary estimates, we further  obtain that we have
\begin{equation*}
\begin{aligned}
\left |
\frac{ |y|(|y + \tau_k z|   - |y|)  - \tau_k   y z  }{(|y + \tau_k z| + |y|)^2 }
\right |
&\leq
\frac{ |y + \tau_k z|  + |y|   + \tau_k  |z| }{|y + \tau_k z| + |y|  }
\leq 2
\end{aligned}
\end{equation*}
and
\begin{equation*}
\begin{aligned}
\left (
 \frac{ 4 y^2  }{(|y| + |y + \tau_k z|)^2}
 \right ) 
\int_{0}^{1} ( 1 -s)  \left( \frac{| y |}{(1-s) | y |  + s | y + \tau_k z |}   \right)^{1-\alpha} \dd s 
&\leq
4 
\int_{0}^{1} (1-s)^{\alpha} \dd s 
\\
&\leq
\frac{4}{\alpha + 1}
\end{aligned}
\end{equation*}
and
\begin{equation*}
\begin{aligned}
&\left  |
  \frac{ 4\tau_k y z    +   \tau_k^2 z^2 }{(|y| + |y + \tau_k z|)^2}
 \right | \int_{0}^{1} ( 1 -s)  \left( \frac{| y |}{(1-s) | y |  + s | y + \tau_k z |}   \right)^{1-\alpha} \dd s
 \\
 &\leq
 \ 
  \frac{ 4\tau_k |y| |z|    +   \tau_k^2 z^2 }{(|y| + |y + \tau_k z|)^2}
  \int_{0}^{1} (1-s)^{\alpha} \dd s 
  \\
  &\leq
  \frac{5}{\alpha + 1}
\end{aligned}
\end{equation*}
a.e.\ in  $\{y \neq 0\}$. As $z=0$ holds a.e.\ in $\{y=0\}$
and since  $z^2|y|^{\alpha - 1}$ is  an element of  $L^1(\{y \neq 0\})$ by our assumption $z \in V_y$, 
the above allows us to 
invoke Lebesgue's dominated convergence theorem 
\cite[Theorem~2.8.1]{Bocharev2007}
to obtain 
\begin{equation*}
\begin{aligned}
&\lim_{k \to \infty}
\frac{1}{(\alpha+1)}
\int_\Omega
\frac{1}{\tau_k}
\left (
\frac{| y + \tau_k z|^{\alpha +1} - | y|^{\alpha+1}}{\tau_k}
-
(\alpha+1)
\sgn(y)|y|^{\alpha}  z
\right )
\dd x
\\
&=
\lim_{k \to \infty}
\frac{1}{(\alpha+1)}
\int_{\{y \neq 0\}}
\frac{1}{\tau_k}
\left (
\frac{| y + \tau_k z|^{\alpha +1} - | y|^{\alpha+1}}{\tau_k}
-
(\alpha+1)
\sgn(y)|y|^{\alpha}  z
\right )
\dd x
\\
&=
\frac{\alpha}{2}
\int_{\{y \neq 0\}}
\frac{z^2}{|y|^{1-\alpha}} 
\dd x.
\end{aligned}
\end{equation*}
This completes the proof. 
\end{proof}

By putting everything together, 
we can now pass to the limit in 
\eqref{eq:diffquotVI} to arrive at a VI 
for the weak accumulation points of the difference quotients $\{\delta_\tau\}$ for $\tau \to 0$. 

\begin{proposition}[VI for limits of difference quotients]
\label{prop:limitVI}
Let $u, h \in H^{-1}(\Omega)$ be given. Define $y := S(u)$. 
Suppose that $\delta_k := \delta_{\tau_k}$
are  difference quotients as in \eqref{eq:diffquot_def} 
associated with a sequence  $\{\tau_k\} \subset (0,\infty)$.
Assume that 
$\tau_k \to 0$ holds and that $\delta_k \weakly \delta$ in $H_0^1(\Omega)$ for some 
$\delta \in H_0^1(\Omega)$.
Then $\delta_k$ converges even strongly in $H_0^1(\Omega)$ to $\delta$ and 
$\delta$ is the solution of the variational inequality 
\begin{equation}
\label{eq:derivativeVI}
\begin{aligned}
&\delta \in V_y, 
\\
&( \delta  , z - \delta )_{H_0^1(\Omega)}
+
\frac{\alpha}{2}
\int_{\{y \neq 0\}}
\frac{z^2}{|y|^{1 - \alpha}} \dd x
-
\frac{\alpha}{2}
\int_{\{y \neq 0\}}
\frac{\delta^2}{|y|^{1 - \alpha}} \dd x
\geq 
\langle h,z -\delta \rangle_{H_0^1(\Omega)}
\\
&\hspace{10.3cm}
\forall z \in V_y.
\end{aligned}
\end{equation}
\end{proposition}

\begin{proof}
Due to \eqref{eq:diffquotVI}, we know that 
\begin{equation*}
\begin{aligned}
&
\langle h, \delta_k - z \rangle_{H_0^1(\Omega)}
+
( \delta_k , z)_{H_0^1(\Omega)}
\\
&\quad+
\frac{1}{(\alpha+1)}
\int_\Omega
\frac{1}{\tau_k}
\left (
\frac{| y + \tau_k z|^{\alpha +1} - | y|^{\alpha+1}}{\tau_k}
-
(\alpha+1)
\sgn(y)|y|^{\alpha}  z
\right )
\dd x
\\
&
\geq 
\| \delta_k\|_{H_0^1(\Omega)}^2
+
\frac{1}{(\alpha+1)}
\int_\Omega
\frac{1}{\tau_k}
\left (
\frac{| y + \tau_k \delta_k|^{\alpha +1} - | y|^{\alpha+1}}{\tau_k}
-
(\alpha+1)
\sgn(y)|y|^{\alpha}  \delta_k
\right )
\dd x
\end{aligned}
\end{equation*}
holds for all $k$ and all $z \in V_y$. 
By taking the limes superior for $k\to \infty$ on the left and the right of 
this inequality, by exploiting  the weak lower semicontinuity of the function
$H_0^1(\Omega) \ni v \mapsto \|v\|_{H_0^1(\Omega)}^2\in \R$,
and by  using 
\cref{lem:lim_transition_1,lem:lim_transition_2}, we get 
\begin{equation}
\label{eq:randomeq3636}
\begin{aligned}
&
\langle h, \delta - z \rangle_{H_0^1(\Omega)}
+
( \delta , z)_{H_0^1(\Omega)}
+
\frac{\alpha}{2}
\int_{\{y \neq 0\}}
\frac{z^2}{|y|^{1-\alpha}} 
\dd x
\\
&
\geq 
\frac{\alpha}{2}
\int_{\{y \neq 0\}}
\frac{\delta^2}{|y|^{1 - \alpha}} \dd x
+
\limsup_{k \to \infty}
\| \delta_k\|_{H_0^1(\Omega)}^2
\\
&
\geq 
\frac{\alpha}{2}
\int_{\{y \neq 0\}}
\frac{\delta^2}{|y|^{1 - \alpha}} \dd x
+
\liminf_{k \to \infty}
\| \delta_k\|_{H_0^1(\Omega)}^2
\\
&
\geq 
\frac{\alpha}{2}
\int_{\{y \neq 0\}}
\frac{\delta^2}{|y|^{1 - \alpha}} \dd x
+
\| \delta \|_{H_0^1(\Omega)}^2
\qquad \forall z \in V_y.
\end{aligned}
\end{equation}
This shows that $\delta$ satisfies \eqref{eq:derivativeVI}
and, since the choice $z=\delta \in V_y$ is allowed in \eqref{eq:randomeq3636} by \cref{lem:lim_transition_1},
that $\|\delta_k\|_{H_0^1(\Omega)} \to \|\delta\|_{H_0^1(\Omega)}$ holds. 
In combination with the weak convergence $\delta_k \weakly \delta$ in $H_0^1(\Omega)$
and the binomial identities, it now follows immediately
 that $\delta_k$ converges also strongly in 
$H_0^1(\Omega)$ to $\delta$  and the proof is complete.
\end{proof}

As a straightforward consequence of \cref{prop:limitVI}, we now obtain our first main result.

\begin{theorem}[G\^{a}teaux differentiability of the solution operator $\boldsymbol{S}$]
\label{th:Sgateaux}
The solution operator 
$S\colon H^{-1}(\Omega) \to H_0^1(\Omega)$ of the PDE 
\eqref{eq:govPDE} is G\^{a}teaux differentiable at all 
points $u \in H^{-1}(\Omega)$. Further,
the G\^{a}teaux derivative $S'(u) \in \LL(H^{-1}(\Omega), H_0^1(\Omega))$
of $S$ at a point $u \in H^{-1}(\Omega)$ with state $y:=S(u)$ is precisely the solution 
operator $H^{-1}(\Omega) \ni h \mapsto \delta \in V_y \subset H_0^1(\Omega)$ 
of the variational problem 
\begin{equation}
\label{eq:DirDiffPDE}
\delta \in V_y,
\qquad 
(\delta, z)_{V_y}
=
\langle h,z \rangle_{H_0^1(\Omega)}
\quad \forall z \in V_y.
\end{equation}
\end{theorem}

\begin{proof}
Let $u, h \in H^{-1}(\Omega)$ be fixed and define $y := S(u)$.
From \eqref{eq:SLipschitz-SS}, we obtain that the 
family of difference quotients $\{\delta_\tau\}$ associated with $u$ and $h$ is bounded in $H_0^1(\Omega)$ 
and, thus, possesses weak accumulation points in $H_0^1(\Omega)$ for $\tau \to 0$.
From \cref{prop:limitVI}, we obtain that all of these 
weak accumulation points are, in fact, strong accumulation
points of the family $\{\delta_\tau\}$ for $\tau \to 0$ and 
solutions of the variational inequality \eqref{eq:derivativeVI}. 
As \eqref{eq:derivativeVI} can have at most one solution
(cf.\  \cite[Theorem 4.1]{Glowinski2013}),
it follows by contradiction that there can only be one accumulation point $\delta$, 
that $\delta_\tau \to \delta$ holds in $H_0^1(\Omega)$ for $\tau \to 0$,
and that $\delta$ is uniquely characterized by \eqref{eq:derivativeVI}. 
This shows that $S$ is directionally differentiable at $u$ in direction $h$
with directional derivative $\delta$. 
By choosing test functions of the 
form 
$\delta + s z$ in \eqref{eq:derivativeVI} for arbitrary $s \in (0,1)$ 
and $z \in V_y$,
we obtain the inequality 
\begin{equation*}
\begin{aligned}
&( \delta  ,  s z )_{H_0^1(\Omega)}
+
\frac{\alpha}{2}
\int_{\{y \neq 0\}}
\frac{(\delta + s z)^2}{|y|^{1 - \alpha}} \dd x
-
\frac{\alpha}{2}
\int_{\{y \neq 0\}}
\frac{\delta^2}{|y|^{1 - \alpha}} \dd x
\geq 
\langle h,  s z \rangle_{H_0^1(\Omega)}
\end{aligned}
\end{equation*}
which, by means of the binomial identities, can also be written as 
\begin{equation*}
\begin{aligned}
&s \left [ ( \delta  ,  z )_{H_0^1(\Omega)}
+
 \alpha 
\int_{\{y \neq 0\}}
\frac{  \delta  z}{|y|^{1 - \alpha}} \dd x
+ 
\frac{\alpha s}{2}
\int_{\{y \neq 0\}}
\frac{z^2}{|y|^{1 - \alpha}} \dd x
\right ]
\geq s
\langle h,   z \rangle_{H_0^1(\Omega)}.
\end{aligned}
\end{equation*}
If we divide by $s$ in the above, let $s$ go to zero,
exploit that $z$ was an arbitrary element of the vector space $V_y$,
and plug in the definition of $(\cdot, \cdot)_{V_y}$, then 
we obtain that $\delta$ satisfies \eqref{eq:DirDiffPDE} as desired.
As \eqref{eq:DirDiffPDE} is precisely the variational problem 
that characterizes the Riesz representative of 
$h \in H^{-1}(\Omega) \subset V_y^*$ 
in the Hilbert space $V_y$ and since $V_y \hookrightarrow H_0^1(\Omega)$ holds,
it follows that the directional derivative $\delta \in H_0^1(\Omega)$
depends linearly and continuously on the direction $h \in H^{-1}(\Omega)$.
Thus, $S$ is G\^{a}teaux differentiable
as claimed, 
with the 
G\^{a}teaux derivative being characterized by \eqref{eq:DirDiffPDE}.
This completes the proof. 
\end{proof}

Due to the Lipschitz estimate for $S$ in \eqref{eq:SLipschitz-SS}
and the compactness of the embedding 
$L^q(\Omega) \hookrightarrow H^{-1}(\Omega)$ 
for all $ \max\left (1, 2d/(d+2) \right ) < q \leq \infty$,
\cref{th:Sgateaux} also implies 
the following Fr\'{e}chet differentiability result for $S$. 

\begin{corollary}[Fr\'{e}chet differentiability of the solution operator $\boldsymbol{S}$]
\label{cor:Frechet}
The solution map
$S$ of the PDE 
\eqref{eq:govPDE} is Fr\'{e}chet   differentiable 
as a function from $L^q(\Omega)$ to $H_0^1(\Omega)$ for all 
$ \max\left (1, 2d/(d+2) \right ) < q \leq \infty$. Its Fr\'{e}chet derivatives 
$S'(u) \in \LL(L^q(\Omega), H_0^1(\Omega))$
are characterized by the variational identity 
\eqref{eq:DirDiffPDE}
in \cref{th:Sgateaux}.
\end{corollary}

\begin{proof}
The assertion follows from a classical contradiction argument. 
We begin with the case $q \neq \infty$. 
Suppose that there are $\max\left (1, 2d/(d+2) \right ) < q < \infty$ 
and $u \in L^q(\Omega)$ such that 
$S$ is not Fr\'{e}chet differentiable at $u$ as a function $S\colon L^q(\Omega) \to H_0^1(\Omega)$.
Then there exist a number $\varepsilon > 0$ and 
sequences $\{h_k\} \subset L^q(\Omega)$ and $\{\tau_k\} \subset (0,\infty)$ such that 
$\|h_k\|_{L^q(\Omega)} = 1$ 
for all $k$, $\tau_k \to 0$ for $k \to \infty$,
and
\[
\varepsilon \leq 
\liminf_{k \to \infty}
\frac{\|S(u + \tau_k h_k) - S(u) - \tau_k S'(u) h_k\|_{H_0^1(\Omega)}}{\tau_k}.
\]
As $\{h_k\} \subset L^q(\Omega)$  is bounded and $q \in (1,\infty)$, we 
may assume without loss of generality that $h_k \weakly h$ holds in $L^q(\Omega)$
for some $h \in L^q(\Omega)$. Due to the compactness 
of the embedding $L^q(\Omega) \hookrightarrow H^{-1}(\Omega)$,
\eqref{eq:SLipschitz-SS},
the G\^{a}teaux differentiability of $S$
at $u$ in direction $h$ obtained from \cref{th:Sgateaux}, 
and $S'(u) \in \LL(H^{-1}(\Omega), H_0^1(\Omega))$, 
it follows that 
\begin{equation*}
\begin{aligned}
\varepsilon 
&\leq
\liminf_{k \to \infty}
\frac{\|S(u + \tau_k h_k) - S(u) - \tau_k S'(u) h_k\|_{H_0^1(\Omega)}}{\tau_k}
\\
&\leq
\liminf_{k \to \infty}
\Bigg (
\frac{\|S(u + \tau_k h) - S(u) - \tau_k S'(u) h\|_{H_0^1(\Omega)}}{\tau_k}
\\
&\hspace{1,7cm}
+
\frac{\|S(u + \tau_k h_k) - S(u + \tau_k h)\|_{H_0^1(\Omega)}}{\tau_k}
+
 \|  S'(u)h - S'(u) h_k \|_{H_0^1(\Omega)} 
\Bigg)
\\
&\leq
0 +
\liminf_{k \to \infty} \| h_k - h\|_{H^{-1}(\Omega)} + 0
=0. 
\end{aligned}
\end{equation*}
This produces the contradiction $0 < \varepsilon \leq 0$ 
and shows that $S$ indeed has to be Fr\'{e}chet differentiable 
as a function $S\colon L^q(\Omega) \to H_0^1(\Omega)$
for all $\max\left (1, 2d/(d+2) \right ) < q < \infty$. 
To obtain the assertion for $q=\infty$,
it now suffices to use H\"older's inequality. 
This completes the proof. 
\end{proof}

We conclude this section with some comments.

\begin{remark}
\label{rem:end_sec_4}~
\begin{itemize}
\item Note that neither \cref{th:Sgateaux} nor \cref{cor:Frechet} makes a statement 
about the continuity of the derivative $u \mapsto S'(u)$. 
The question of whether/on which sets the operator $S$ is continuously 
differentiable is inherently related to the questions 
of how the space $V_y$ varies with $y$ and 
of whether the Meyers-Serrin 
identity $H=W$ holds in $V_y$.
We remark that the latter of these two questions is a particularly delicate one for the space $V_y$ as the overwhelming majority of results on the validity of the identity $H=W$ in weighted Sobolev spaces found in the literature require certain integrability assumptions on the weight function, in particular the famous Muckenhoupt $A_p$-property. Compare, for instance, 
with \cite{Zhikov1998} and
\cite[Definition 1.2.2, Corollary 2.1.6]{Turesson2000} in this context. 
As the solution $y$ of \eqref{eq:govPDE}
may be arbitrarily smooth and, thus, may go to zero arbitrarily fast when approaching the zero level set 
$\{ y = 0\}$, one cannot expect that 
the weight  $|y|^{\alpha - 1}\mathds{1}_{\{y \neq 0\}}$  satisfies such a property. As a consequence, 
the verification of the Meyers-Serrin identity $H=W$ for $V_y$ is, 
at least to the author's best knowledge, an open problem.
Compare also with the remarks after \cref{th:KKT} 
in this context. 
We leave this topic for future research. 
\item 
The analysis of this section can be extended to semilinear elliptic PDEs
whose nonlinearities are induced by continuous, nondecreasing functions $g\colon \R \to \R$
that are smooth except for finitely many, root-like cusps. 
For such semilinear equations, one has to employ a Taylor approximation near 
the cusps to arrive at results analogous to \cref{lem:lim_transition_1,lem:lim_transition_2};
cf.\ \cite[Proof of Theorem 4.2.3]{ChristofPhd2018}, where such a calculation is performed. 
The main argument, 
that has to work for 
the  analysis of this section to go through, 
is that the boundedness of the second-order difference quotients involving $\delta_\tau$
in the elliptic VI of the second kind satisfied by $\delta_\tau$ imposes
conditions on the weak accumulation 
points of the family $\{\delta_\tau\}$
that, when assumed for a test function $z$,
allow to perform the limit transition with the second-order 
difference quotients when $\delta_\tau$ is replaced by $z$.
See \cref{lem:lim_transition_1,lem:lim_transition_2}, where we have performed precisely 
these two steps. As already mentioned in the introduction, 
the method of proof that we have used in this section
can also be systematized on an abstract level, 
using the concepts of twice epi-differentiability
and second-order subderivative;
see \cite[Chapter~1]{ChristofPhd2018}.
By working with these more involved concepts and 
deeper results, it is also possible to obtain analogues
of \cref{th:Sgateaux} and  \cref{cor:Frechet} for PDEs of the form 
\begin{equation}
\label{eq:q-laplace_eq}
-\Delta y -   \nabla \cdot \left( \frac{\nabla y}{|\nabla y|^{1-\alpha}} \right )
= u \text{ in } \Omega,
\quad y = 0 \text{ on } \partial \Omega,
\end{equation}
with $\alpha \in (0,1)$; see \cite[Section 4.3.5]{ChristofPhd2018}.
We focus on the model equation
\mbox{$-\Delta y 
+ \sgn(y)|y|^{\alpha} = u$}
in this paper because 
this allows us to obtain the differentiability results 
in \cref{th:Sgateaux} and \cref{cor:Frechet} with
rather elementary means and, thus, 
without obscuring the basic ideas of the
analysis and the underlying mechanisms with technicalities.
\end{itemize}
\end{remark}

\section{Necessary optimality conditions for the optimal control problem}
\label{sec:5}

With the differentiability results in 
\cref{th:Sgateaux,cor:Frechet}
at hand, we can turn our attention to 
the derivation of first-order necessary optimality conditions for the optimal 
control problem 
\begin{equation}
\tag{P}
\left.
\begin{aligned}
\text{Minimize }& J(y,u)
\\
\text{w.r.t.\ } &y \in H_0^1(\Omega),\quad u \in L^2(\Omega),
\\
\text{ s.t.\ } & 
-\Delta y 
+ \sgn(y)|y|^{\alpha} = u \text{ in } \Omega,
\quad y = 0 \text{ on } \partial \Omega,
\\
\text{ and } & u \in \Uad.
\end{aligned}
~~\right \}
\end{equation}
We begin by recalling our assumptions 
on the quantities in \eqref{eq:P}.

\begin{assumption}[standing assumptions on (P)]\label{ass:standing:P}~
\begin{enumerate}[label=\roman*)]
\item $\Omega \subset \R^d$, $d \in \N$, is a nonempty open bounded set;
\item $J\colon H_0^1(\Omega) \times L^2(\Omega) \to \R$ is a Fr\'{e}chet differentiable function;
\item $\alpha \in (0,1)$ is a given exponent;
\item $\Uad$ is a nonempty and convex subset of $L^2(\Omega)$.
\end{enumerate}
\end{assumption}

Note that, to be able to ensure that \eqref{eq:P} possesses an optimal control-state pair, 
one requires more assumptions on 
 $\Uad$ and $J$ than stated in \cref{ass:standing:P}.
 For the discussion of necessary optimality conditions, 
 however, this is irrelevant. As a first consequence 
 of \cref{cor:Frechet}, we obtain 
 the following result. 

\begin{proposition}[Bouligand stationarity condition]
\label{prop:Bouligand}
Suppose that $\bar u \in L^2(\Omega)$ is a locally optimal 
control of \eqref{eq:P} with associated state $\bar y := S(\bar u)$, i.e., 
an $L^2(\Omega)$-local minimizer of the reduced problem
\[
\text{Minimize}~J(S(u),u)
\quad 
\text{s.t.}
\quad u \in \Uad.
\]
Then $\bar u$ is a solution of the variational inequality
\begin{equation}
\label{eq:Bouligand_condition}
\bar u \in \Uad,
\quad 
\left \langle \partial_y J(\bar y, \bar u), S'(\bar u)(u - \bar u) \right \rangle_{H_0^1(\Omega)}
+
\left ( \partial_u J(\bar y, \bar u), u - \bar u \right )_{L^2(\Omega)}
\geq 0
~~\forall u \in \Uad.
\end{equation}
Here, $ \partial_y J(\bar y, \bar u) \in H^{-1}(\Omega)$ and 
$\partial_u J(\bar y, \bar u) \in L^2(\Omega) \cong L^2(\Omega)^*$ 
denote the partial Fr\'{e}chet derivatives of $J$
w.r.t.\ $y$ and $u$ at $(\bar y, \bar u)$, respectively. 
\end{proposition}

\begin{proof}
The assertion of this proposition follows immediately from \cref{cor:Frechet},
the Fr\'{e}chet differentiability of $J$, the chain rule, 
and the convexity of $\Uad$.
\end{proof}

By exploiting the characterization of $ S'(\bar u)$ by means of 
the variational identity \eqref{eq:DirDiffPDE},
we can reformulate 
\eqref{eq:Bouligand_condition} 
to arrive at a more tangible optimality condition. 
The resulting stationarity system is
the second main result of this work. 

\begin{theorem}[KKT-system]
\label{th:KKT}
Let $\bar u \in L^2(\Omega)$ be a control  with associated state $\bar y := S(\bar u)$.
Then $\bar u$ satisfies the Bouligand stationarity condition 
\eqref{eq:Bouligand_condition} 
if and only if 
there exists an adjoint state $\bar p$ 
such that the following KKT-system holds:
\begin{equation}
\label{eq:KKT}
\begin{gathered}
\bar u \in \Uad,\qquad \bar y \in H_0^1(\Omega),
\qquad \bar p \in V_{\bar y},
\\
-\Delta \bar y  
+  \sgn(\bar y)|\bar y|^{\alpha} = \bar u \text{ in } \Omega,
\quad \bar y = 0 \text{ on } \partial \Omega,
\\
(\bar p, z)_{H_0^1(\Omega)}
+
\alpha
\int_{\{\bar y  \neq 0\}} \frac{\bar p z}{|\bar y|^{1 - \alpha}} \dd x
=
\left \langle \partial_y J(\bar y, \bar u), z \right \rangle_{H_0^1(\Omega)}
\quad \forall z \in V_{\bar y},
\\
\left (
\bar p + \partial_u J(\bar y, \bar u),
u - \bar u
\right )_{L^2(\Omega)} \geq 0
\quad \forall u \in \Uad.
\end{gathered}
\end{equation}
Here, the state equation is again understood weakly and $V_{\bar y}$ is defined  
by 
\begin{equation}
\label{eq:Vy_def_again}
V_{\bar y}
:=
\Bigg \{
v \in H_0^1(\Omega)
\colon
\int_{\{\bar y  \neq 0\}} \frac{v^2}{|\bar y|^{1 - \alpha}} \dd x
< \infty
~
\text{and }
v=0 \text{ a.e.\ in } \{\bar y  = 0\}
\Bigg \}.
\end{equation}
\end{theorem}

\begin{proof}
Suppose that $\bar u$ and $\bar y := S(\bar u)$ satisfy \eqref{eq:Bouligand_condition}. 
Then the Riesz representation theorem implies that there exists a unique 
$\bar p$ satisfying 
\begin{equation}
\label{eq:barp_eq}
\bar p \in V_{\bar y},
\qquad 
(\bar p, z)_{V_{\bar y}}
=
\left \langle \partial_y J(\bar y, \bar u), z \right \rangle_{H_0^1(\Omega)}
\quad \forall z \in V_{\bar y}.
\end{equation}
Using \eqref{eq:DirDiffPDE}, we may compute that 
\begin{equation*}
\begin{aligned}
\left \langle \partial_y J(\bar y, \bar u), S'(\bar u)(u - \bar u) \right \rangle_{H_0^1(\Omega)}
&=
(\bar p,  S'(\bar u)(u - \bar u) )_{V_{\bar y}}
\\
&=
\left \langle u - \bar u, \bar p \right \rangle_{H_0^1(\Omega)}
=
(\bar p,   u - \bar u )_{L^2(\Omega)}
\end{aligned}
\end{equation*}
holds for all $u \in \Uad$. 
If we plug this identity into \eqref{eq:Bouligand_condition} 
and collect everything, then 
we obtain that  $\bar u$ and $\bar y$ satisfy \eqref{eq:KKT} as claimed. 

To prove that every tuple $(\bar u, \bar y)$
that satisfies \eqref{eq:KKT} for some $\bar p \in V_{\bar y}$
is a solution of the variational inequality \eqref{eq:Bouligand_condition},
we can proceed along the exact same lines in reverse. 
This establishes the assertion. 
\end{proof}

Note that
 the nondifferentiability  of the function $\R \ni s \mapsto \sgn(s)|s|^\alpha \in \R$
present in \eqref{eq:govPDE} manifests itself in \eqref{eq:KKT} in the form of the 
singular weight $|\bar y|^{\alpha - 1}\mathds{1}_{\{\bar y \neq 0\}}$ 
and the weighted Sobolev space $V_{\bar y}$. 
We remark that the appearance of these quantities 
significantly complicates the 
analysis of the regularity properties of the adjoint state $\bar p$.
In particular, 
it is typically 
completely unclear whether $\bar p$ possesses, e.g.,  some form of $W^{2,q}$-regularity;
cf.\ \cite{Turesson2000,Zhikov1998} and the discussion in \cref{rem:end_sec_4}.
Note that the question of whether  $W^{2,q}$-regularity can be obtained for $\bar p$  is of special interest
as it is crucial for the derivation of a-priori error estimates 
for finite element discretizations of \eqref{eq:P}. 
The only thing that can be  established rather easily for the adjoint state 
$\bar p$ in \eqref{eq:KKT} is improved $L^q(\Omega)$-regularity 
as the following result shows. 

\begin{corollary}[higher Lebesgue regularity for the adjoint state]
\label{cor:higher_reg_adjoint}
Suppose that $\bar u$, $\bar y$, and $\bar p$ satisfying 
\eqref{eq:KKT} are given and that 
$\partial_y J(\bar y, \bar u) \in L^s(\Omega)$ holds for 
some $\max(1, 2d/(d+2)) < s \leq \infty$. Define 
\[
\RR_s
:=
\begin{cases}
[1,\infty] &\text{ if } s > d/2,
\\
[1,\infty) &\text{ if } s = d/2,
\\
\displaystyle
\left [
1, \left (\frac{1}{s} - \frac{2}{d}\right)^{-1} \right )
&\text{ if } s < d/2.
\end{cases}
\]
Then it holds $\bar p \in L^r(\Omega)$ for all $r \in \RR_s$. 
\end{corollary}

\begin{proof}
To establish the asserted higher Lebesgue regularity of $\bar p$, we use 
a classical argumentation of Stampacchia that relies on the analysis
of ``truncations'' of the function $\bar p$, the Sobolev embeddings, 
and the study of the measure of sub- and superlevel sets; see \cite[Appendix II-B]{KinderlehrerStampacchia1980}.
Define $\bar p_k := \bar p - \min(k, \max(-k,\bar p))$, $k \geq 0$.
Then $\bar p_k$ is an element of $V_{\bar y}$ for all $k\geq 0$ since 
$\bar p_k = 0 - \min(k, \max(-k,0)) = 0$ holds a.e.\ in $\{\bar y = 0\}$,
since $\bar p_k \in H_0^1(\Omega)$ holds by \cite[Theorem 5.8.2]{Attouch2006},
and since
\[
\int_{\{\bar y  \neq 0\}} \frac{\bar p_k^2}{|\bar y|^{1 - \alpha}} \dd x
\leq
\int_{\{\bar y  \neq 0\}} \frac{\bar p^2}{|\bar y|^{1 - \alpha}} \dd x
< \infty \qquad \forall k \geq 0. 
\]
In particular, $\bar p_k$ is a valid test function in \eqref{eq:barp_eq}.
In combination with \eqref{eq:sob_emb_primal} and
again \cite[Theorem 5.8.2]{Attouch2006}, 
this allows us to conclude that 
\begin{equation*}
c_q \|\bar p_k \|_{L^q(\Omega)}^2
\leq
\|\bar p_k \|_{H_0^1(\Omega)}^2
\leq
(\bar p_k, \bar p_k )_{V_{\bar y}}
\leq
(\bar p, \bar p_k )_{V_{\bar y}}
=
\int_\Omega
\partial_y J(\bar y, \bar u)
\bar p_k\,
\dd x
\end{equation*}
holds for all $q \in \QQ_d$ with some constants $c_q > 0$ and
 $\QQ_d$ 
as in \eqref{eq:sob_emb_primal}.
To obtain the assertion of the corollary,
it now suffices to invoke 
\cite[Lemma 3.4]{ChristofWachsmuth2023},
which, by means of the aforementioned analysis of 
the measure of the sub- and superlevel sets of $\bar p$,  
establishes that the inequality 
\begin{equation*}
\|\bar p_k \|_{L^q(\Omega)}^2
\leq
\frac{1}{c_q}
\int_\Omega
\left | \partial_y J(\bar y, \bar u)
\bar p_k \right |\,
\dd x\qquad \forall k \geq 0
\qquad \forall q \in \QQ_d
\end{equation*}
with $\partial_y J(\bar y, \bar u) \in L^s(\Omega)$
implies that $\bar p \in L^r(\Omega)$ holds for all $r \in \RR_s$.
\end{proof}

Prototypical examples of functions $J$ that are covered by \cref{cor:higher_reg_adjoint} 
are tracking-type 
objectives of the form 
\[
J\colon H_0^1(\Omega) \times L^2(\Omega) \to \R,
\qquad
J(y,u)
:= 
\frac12 \|y - y_D\|_{L^2(\Omega)}^2
+
\frac{\nu}{2} \|u\|_{L^2(\Omega)}^2,
\]
involving a desired state $y_D \in L^2(\Omega)$ and a Tikhonov 
parameter $\nu > 0$. 
For this type of objective function, we obtain 
from 
\cref{th:KKT} and 
\cref{cor:higher_reg_adjoint} 
that the adjoint state $\bar p$ associated with a Bouligand stationary control $\bar u$
has to satisfy $\bar p \in V_{\bar y} \cap L^r(\Omega)$ for all $r \in \RR_d$, where $\RR_d$ is
given by 
\[
\RR_d
:=
\begin{cases}
[1,\infty] &\text{ if } d < 4,
\\
[1,\infty) &\text{ if } d = 4,
\\
\displaystyle
\left [
1,  \frac{2d}{d-4} \right )
&\text{ if } d>4,
\end{cases}
\]
where $V_{\bar y}$ is defined as in \eqref{eq:Vy_def_again},
and where $d$ is the dimension of $\Omega$.  
If, additionally, we assume that  $\Uad$ has the form 
$\Uad = \{u \in L^2(\Omega) \colon u_a \leq u \leq u_b \text{ a.e.\ in }\Omega\}$
for some $u_a, u_b \in H^1(\Omega)$
satisfying $u_a \leq 0 \leq u_b$ a.e.\ in $\Omega$, 
then the projection formula 
\[
\bar u =
\max\left ( u_a, 
\min\left ( u_b, -\frac{\bar p}{\nu} \right ) \right )
\]
obtained from the last line in \eqref{eq:KKT} and Stampacchia's lemma \cite[Theorem 5.8.2]{Attouch2006}
also imply that $\bar u$ inherits all of these regularity properties from 
$\bar p$, i.e., we have 
\[
\bar u \in V_{\bar y} \cap L^r(\Omega) \qquad \forall r \in \RR_d.
\]
As already mentioned, whether optimal controls $\bar u$ and their associated 
adjoint states $\bar p$ possess higher regularity properties than those above 
in the situation of \eqref{eq:P} (e.g., Lipschitz regularity or a form of $W^{2,q}$-regularity)
is an open problem.
The same is true for the derivation of (no-gap) second-order optimality and quadratic growth 
conditions for problems of the type \eqref{eq:P}.
We leave these topics for future research. 


\bibliographystyle{alpha}
\bibliography{references}

\medskip
Received xxxx 20xx; revised xxxx 20xx.
\medskip

\end{document}